\documentclass[12pt]{amsart}

\usepackage{amsmath,amsthm,amssymb}
\usepackage{enumerate}
\usepackage{tikz}
\usetikzlibrary{positioning}
\usepackage{hyperref}  
\usepackage{amsfonts,amscd}
\usepackage{mathtools}
\usepackage{wasysym} 
\usepackage{graphicx}
\usepackage{psfrag}
\usepackage{datetime}
\usepackage{tikz-cd}  
\usepackage{geometry} 
\usepackage[all]{xy}

\DeclareFontFamily{U}{mathx}{\hyphenchar\font45}
\DeclareFontShape{U}{mathx}{m}{n}{
      <5> <6> <7> <8> <9> <10>
      <10.95> <12> <14.4> <17.28> <20.74> <24.88>
      mathx10
      }{}
\DeclareSymbolFont{mathx}{U}{mathx}{m}{n}
\DeclareFontSubstitution{U}{mathx}{m}{n}
\DeclareMathAccent{\widecheck}{0}{mathx}{"71}

\newtheorem{theorem}{Theorem}[section]
\newtheorem{lemma}[theorem]{Lemma}
\newtheorem{proposition}[theorem]{Proposition}

\newtheorem{innercustomthm}{Theorem}
\newenvironment{customthm}[1]
  {\renewcommand\theinnercustomthm{#1}\innercustomthm}
  {\endinnercustomthm}

\theoremstyle{definition}
\newtheorem{definition}[theorem]{Definition}

\theoremstyle{remark}
\newtheorem{remark}[theorem]{Remark}

\newcommand{\nw}[1]{\textcolor{black}{#1}}    
\newcommand{\jdm}[1]{\textcolor{black}{#1}}   
\newcommand{\nnw}[1]{\textcolor{black}{#1}}    

\title[]{Support $\tau_2$-tilting and 2-torsion pairs}
\author{Jordan McMahon}
\address{unaffiliated}
\email{jordanmcmahon37@gmail.com}

\begin{document}
\begin{abstract}
The theory of $\tau$-tilting was introduced by Adachi--Iyama--Reiten; one of the main results is a bijection between support $\tau$-tilting modules and torsion classes. We are able to generalise this result in the context of the higher Auslander--Reiten theory of Iyama. For a finite-dimensional algebra $A$ with 2-cluster-tilting subcategory $\mathcal{C}\subseteq\mathrm{mod}A$, we are able to find a correspondence between support $\tau_2$-tilting $A$-modules and torsion pairs in $\mathcal{C}$ satisfying an additional functorial finiteness condition. 
\end{abstract}

\maketitle

\section{Introduction}

Support-tilting modules were first studied by Ringel \cite{ring2} and connected to torsion classes and cluster algebras by Ingalls--Thomas \cite{it}. Since support-tilting modules are able to capture the behaviour of clusters, this led to further study. Significantly, Adachi, Iyama and Reiten \cite{adir} introduced support $\tau$-tilting modules, and were able to find a bijection, for an finite-dimensional algebra $A$, between support $\tau$-tilting $A$-modules and functorially-finite torsion classes in $\mathrm{mod}A$. See also \cite{ir} for a survey of $\tau$-tilting theory, which has seen much activity in recent years \cite{asai}, \cite{bst}, \cite{br}, \cite{irrt}, \cite{irtt}, \cite{pppp}, \cite{pla} as well as its generalisations such as in silting theory \cite{ahmv}, \cite{ahs2}, \cite{bl}.
A natural question to ask is whether similar results are true in the context of higher Auslander--Reiten theory, as introduced by Iyama in \cite{iy3}, \cite{iy1}, and an active area of research \cite{di}, \cite{is}, \cite{jk}, \cite{kva}, \cite{mcmahon}, \cite{ps}. As the name suggests, higher Auslander--Reiten theory has connections to higher-dimensional geometry and topology \cite{djl}, \cite{djw}, \cite{mw}, \cite{ot}, \cite{will} and is hence a natural generalisation. The result we find is the following:

\begin{theorem}\label{elso}
Let $A$ be a finite-dimensional algebra with $2$-cluster-tilting subcategory $\mathcal{C}\subseteq\mathrm{mod}A$. Then there is a correspondence between support $\tau_2$-tilting $A$-modules and 2-functorially finite torsion pairs in $\mathcal{C}$.
\end{theorem}

\section{Background and Notation}
Consider a finite-dimensional algebra $\nnw{A}$ over a field $K$, and fix a positive integer $d$. We will assume that $\nnw{A}$ is of the form $KQ/I$, where $KQ$ is the path algebra over some quiver $Q$ and $I$ is an admissible ideal of $KQ$.
An $\nnw{A}$-module will mean a finitely-generated left $\nnw{A}$-module; by $\mathrm{mod}\nnw{A}$ we denote the category of $\nnw{A}$-modules. The functor $D=\mathrm{Hom}_K(-,K)$ defines a duality. Let $\mathrm{add}M$ be the full subcategory of $\mathrm{mod}\nnw{A}$ composed of all $\nnw{A}$-modules isomorphic to direct summands of finite direct sums of copies of $M$. For an $A$-module $M$, the right annihilator of $M$ is the two-sided ideal $\mathrm{ann}M:=\{a\in A|aM=0\};$ for a class of modules $\mathcal{X}$ we set $\mathrm{ann}\mathcal{X}:=\{a\in A|aX=0\ \forall\ X\in\mathcal{X}\}.$ For an $A$-module $M$, define $\mathrm{Sub}M:=\{N\in\mathrm{mod}A|\exists\ \text{injection}\ N\hookrightarrow M\};$ for a class of modules $\mathcal{X}$ we set $\mathrm{Sub}\mathcal{X}:=\{N\in\mathcal{X}|\exists X\in\mathcal{X},\ \text{injection}\ \ N\hookrightarrow X\}.$ Define $\mathrm{Fac}M$ and $\mathrm{Fac}\mathcal{X}$ dually.

\subsection{\jdm{Higher} cluster-tilting subcategories}\label{precl-sect}

Define  $\tau_d:=\tau\Omega^{d-1}$ to be the \emph{$d$-Auslander--Reiten translation} and  $\tau_{d}^-:=\nw{\tau^{-}\Omega^{-(d-1)}}$ to be the \emph{inverse $d$-Auslander--Reiten translation}.
A subcategory $\mathcal{C}$ of $\mathrm{mod}A$ is \emph{precovering} or \emph{contravariantly finite} if for any $M\in \mathrm{mod}\Lambda$ there is an object $C_M\in\mathcal{C}$ and a morphism $f:C_M\rightarrow M$ such that
$\mathrm{Hom}(C,-)$ is exact on the sequence 
$$\begin{tikzcd}
C_M\arrow{r}&M\arrow{r}&0
\end{tikzcd}$$
for all $C\in\mathcal{C}$. The module $C_M$ is said to be a \emph{right $\mathcal{C}$-approximation}. The dual notion of precovering is \emph{preenveloping} or \emph{covariantly finite}.  A subcategory $\mathcal{C}$ that is both precovering and preenveloping is called \emph{functorially finite}.
\begin{definition}\cite[Definition 2.2]{iy1}
For a finite-dimensional algebra $\nnw{A}$, a subcategory $\mathcal{C}\subseteq\mathrm{mod}(\nnw{A})$ is a \emph{$d$-cluster-tilting subcategory} if it satisfies the following conditions:
\begin{align*}
\mathcal{C}&=\{X\in\mathcal{C}\mid \mathrm{Ext}^i_{\nnw{A}}(M,X)=0\ \forall\ 0< i<d \}.\\
\mathcal{C}&=\{X\in\mathcal{C} \mid \mathrm{Ext}^i_{\nnw{A}}(X,M)=0\ \forall\ 0< i<d \}.
\end{align*} 
\end{definition}
 A \emph{right $\mathcal{C}$-resolution} is a sequence
$$\begin{tikzcd}
\cdots \arrow{r}&C_1\arrow{r}&C_0\arrow{r}&M\arrow{r}&0
\end{tikzcd}$$
with $C_i\in\mathcal{C}$ for each $i$, and which becomes exact under $\mathrm{Hom}_A(C,-)$ for each $C\in\mathcal{C}$. Define a \emph{left $\mathcal{C}$-resolution} dually.

\begin{theorem}\cite[Theorem 3.6.1]{iy1}\label{iyamatheorem}
Let $\mathcal{C}\subseteq \mathrm{mod} A$ be a $d$-cluster-tilting subcategory. Then 
\begin{enumerate}
\item Any $M\in\mathrm{mod}A$ has a right $\mathcal{C}$-resolution
$$\begin{tikzcd}
0\arrow{r}&C_{d-1}\arrow{r}&\cdots \arrow{r}&C_1\arrow{r}&C_0\arrow{r}&M\arrow{r}&0.
\end{tikzcd}$$
\item Any $M\in\mathrm{mod}A$ has a left $\mathcal{C}$-resolution
$$\begin{tikzcd}
0\arrow{r}&M \arrow{r}&C_0\arrow{r}&C_1\arrow{r}&\cdots\arrow{r}&C_{d-1}\arrow{r}&0.
\end{tikzcd}$$
\end{enumerate}
\end{theorem}

Recall the stable module category $\underline{\mathrm{mod}}A$ is full subcategory of $\mathrm{mod}A$ obtained by factoring out all morphisms that factor through a projective module.
\begin{theorem}\cite[Theorem 1.5]{iy1}\label{lada}
Let $A$ be a finite-dimensional algebra. Then: 
\begin{itemize}
\item If $\mathrm{Ext}^i_A(M,A)=0$ for all $0<i<d$, then $\mathrm{Ext}^i_A(M,N)\cong D\mathrm{Ext}^{d-i}_A (N,\tau_dM)$ and $\underline{\mathrm{Hom}}_A(M,N)\cong D\mathrm{Ext}^d_A(N,\tau_dM)$ for all $M\in\mathrm{mod}A$ and all $0<i<d$.
\item If $\mathrm{Ext}^i_A(DA,N)=0$ for all $0<i<d$, then $\mathrm{Ext}^i_A(M,N)\cong D\mathrm{Ext}^{d-i}_A (\tau_d^-N,M)$ and $\overline{\mathrm{Hom}}_A(M,N)\cong D\mathrm{Ext}^d_A(\tau_d^-N,M)$ for all $N\in\mathrm{mod}A$ for all $0<i<d$.
\end{itemize}
\end{theorem}

We may now generalise a result of Auslander--Smal{\o}, the proof of which is unchanged apart from indices.

\begin{proposition}[c.f. Proposition 5.8 of \cite{as}]\label{aslemma}
Let $A$ be a finite-dimensional algebra. Then for $X,Y\in\mathrm{mod}A$ the following are equivalent:
\begin{enumerate}[(i)]
\item  $\mathrm{Hom}_A(\tau_2^-Y,X)=0$.
\item $\underline{\mathrm{Hom}}_A(\tau_2^-Y,\mathrm{Sub}X)=0$ .
\item $\mathrm{Ext}^2_A(\mathrm{Sub}X,Y)=0$.
\end{enumerate}
\end{proposition}

\begin{proof}
Firstly, statements $(ii)$ and $(iii)$ are equivalent by Theorem \ref{lada} and statement $(i)$ trivially implies $(ii)$. It remains to show $(ii)$ implies $(i)$; we prove by contradiction. So suppose there is a morphism $f:\tau_2^-Y\rightarrow X$ is a non-zero morphism with image $\mathrm{Im}f \in\mathrm{Sub} X$, and such that the induced morphism $f^\prime:\tau_2^-Y\twoheadrightarrow \mathrm{Im}f$ factors through a projective module $P$. Since $f^\prime$ is surjective, we may assume $P$ is the projective cover of $\mathrm{Im}f$. Let $g:\tau_2^-M\rightarrow P$ be any morphism; since $\tau_2^-Y$ has no non-trivial projective summands, we have $g(\tau_2^-M)\subset \mathrm{rad}P$, hence any composition $\tau_2^-Y\rightarrow P\rightarrow \mathrm{Im}f$ is not an epimorphism. Therefore the image of $f^\prime$ in $\underline{\mathrm{Hom}}_A(\tau_2^-Y,\mathrm{Im}f)=0$ is not zero. Hence $\mathrm{Hom}_A(\tau_2^-Y,X)=0$ if and only if $\underline{\mathrm{Hom}}_A(\tau_2^-Y,\mathrm{Sub}X)=0$.
\end{proof}
Key homological tools for higher cluster-tilting categories are $d$-pushouts and $d$-pullbacks, constructed as follows.

\begin{proposition}\cite[Proposition 3.8]{jasso}\label{jasso}
Let $A$ be a finite-dimensional algebra and $\mathcal{C}\subseteq\mathrm{mod}A$. For any $d$-exact sequence in $\mathcal{C}$
$$0\rightarrow Y_0\rightarrow Y_1\rightarrow\cdots \rightarrow Y_{d+1}\rightarrow 0$$ and any morphism $f:X_{d+1}\rightarrow Y_{d+1}$ there exists a commutative diagram in $\mathcal{C}$:

$$\begin{tikzcd}
0\arrow{r}&Y_0\arrow{r}\arrow[equals]{d}&X_1\arrow{r}\arrow{d}&\cdots \arrow{r}&X_{d}\arrow{d}\arrow{r}&X_{d+1}\arrow{r}\arrow{d}{f}&0\\
0\arrow{r}&Y_0\arrow{r}&Y_1\arrow{r}&\cdots \arrow{r}&Y_{d}\arrow{r}&Y_{d+1}\arrow{r}&0
\end{tikzcd}$$
such that there is an induced $d$-exact sequence
$$0\rightarrow X_1\rightarrow X_2\oplus Y_1\rightarrow \cdots\rightarrow X_{d+1}\oplus Y_{d}\rightarrow Y_{d+1}\rightarrow 0$$
\end{proposition}
The construction of this commutative diagram is obtained inductively as follows: $X_d$ is defined to be the right $\mathcal{C}$-approximation of the pullback of $((X_{d+1}\oplus Y_d)\rightarrow Y_{d+1})$. Subsequently $X_{d-1}$ is defined to be the right $\mathcal{C}$-approximation of the pullback of $((X_{d}\oplus Y_{d-1})\rightarrow Y_{d})$. This continues until $X_0=Y_0$ is reached.

\begin{remark}\label{pullre}
In this setting, it is actually immaterial whether $Y_{d+1}$ is in $\mathcal{C}$ or not: Theorem \ref{iyamatheorem} ensures that the so-called $d$-kernel of the morphism $(X_{d+1}\oplus Y_d)\rightarrow Y_{d+1})$ exists. Equally, it is not important whether $Y_0\in\mathcal{C}$.
\end{remark} A technical result is the following:
\begin{lemma}\label{pushout}
Consider a module $X\in\mathcal{C}$ and two $d$-exact sequences:
$$0\rightarrow L\rightarrow M\rightarrow N\rightarrow X\rightarrow 0,$$
$$0\rightarrow L^\prime\rightarrow M^\prime\rightarrow N^\prime \rightarrow X\rightarrow 0.$$
Then there exist $P,Q\in\mathcal{C}$ with no common non-zero summands and a commutative diagram with exact rows and columns:

$$\begin{tikzcd}
&&&0\arrow{d}&0\arrow{d}\\
&&&L\arrow[equal]{r}\arrow{d}& L\arrow{d}\\
&& Q\arrow{r}\arrow{d}& Q\oplus M \arrow{r}\arrow{d}& M\arrow{d} \arrow{r}&0\\
0\arrow{r}&L^\prime \arrow{r}\arrow[equal]{d}&Q\oplus M^\prime \arrow{d}\arrow{r}&P\arrow{r}\arrow{d}&N\arrow{r}\arrow{d}&0\\
0\arrow{r} &L^\prime \arrow{r}&M^\prime \arrow{r}\arrow{d}&N^\prime\arrow{r}\arrow{d}&X\arrow{r}\arrow{d}&0\\
&&0&0&0
\end{tikzcd}$$ 
\end{lemma}

\begin{proof}
Just as in the above method for $d$-pullbacks, we may form the following commutative diagram with exact rows and columns by taking the appropriate right $\mathcal{C}$-approximations of pullbacks: $P,Q,R,S$
$$\begin{tikzcd}
&&&0\arrow{d}&0\arrow{d}\\
&&&L\arrow[equal]{r}\arrow{d}& L\arrow{d}\\
&& Q\arrow{r}\arrow{d}& S \arrow{r}\arrow{d}& M\arrow{d}\arrow{r}&0\\
0\arrow{r}&L^\prime \arrow{r}\arrow[equal]{d}&R\arrow{d}\arrow{r}&P\arrow{r}\arrow{d}&N\arrow{r}\arrow{d}&0\\
0\arrow{r} &L^\prime \arrow{r}&M^\prime \arrow{r}\arrow{d}&N^\prime\arrow{r}\arrow{d}&X\arrow{r}\arrow{d}&0\\
&&0&0&0
\end{tikzcd}$$ 
It suffices to show the surjection $R\twoheadrightarrow M^\prime$ is split (by symmetry, also $S\twoheadrightarrow M$ splits). By the pullback property of $P$, we find $M^\prime\rightarrow N^\prime$ factors through $P$. But the pullback property of $R$ then implies both $M^\prime\rightarrow P$ factors through $R$ and that $M^\prime\rightarrow R\rightarrow M^\prime \cong \mathrm{id}$. Hence $R\cong M^\prime \oplus Q$.

Finally, suppose $P$ and $Q$ have a common summand $Y$. Then any morphism $Y\rightarrow N^\prime$ must factor through $M^\prime$. But as above, $M^\prime \rightarrow N^\prime$ factors through $P$. Hence $Y$ must be a summand of $M^\prime$. This is a contradiction, since $M^\prime$ is already included a summand of $Q\oplus M^\prime$, and hence cannot be a summand of $Q$. This finishes the proof. 
\end{proof}
A second technical result will prove similarly useful.
\begin{lemma}\label{topushout}
In the setting of Lemma \ref{pushout}, in case there exists $P^\prime \in\mathcal{C}$ such that $P\cong P^\prime \oplus M^\prime$, then there is a commutative diagram with exact rows and columns

$$\begin{tikzcd}
&&0\arrow{d}&0\arrow{d}\\
&0\arrow{r}& L\arrow{r}\arrow{d}& M \arrow{dr}\arrow{d}\\
0\arrow{r}&L^\prime \arrow{r}\arrow[equal]{d}&Q  \arrow{d}\arrow{r}&P^\prime\arrow{r}\arrow{d}&N\arrow{r}\arrow{d}&0\\
0\arrow{r} &L^\prime \arrow{r}&M^\prime \arrow{r}&N^\prime\arrow{r}&X\arrow{r}\arrow{d}&0\\
&&&&0
\end{tikzcd}$$ 
\end{lemma}
\begin{proof}
In this case we know that the morphism $L^\prime\rightarrow M^\prime$ factors through $Q$. Since we are dealing with a $2$-pullback diagram, the $2$-exact sequence:
$$0\rightarrow L\rightarrow M\oplus Q\rightarrow P^\prime \oplus M^\prime \rightarrow N^\prime\rightarrow 0$$ may simply be transcribed into a commutative diagram as stated.
\end{proof}
\section{Main results}
In this section let $\mathcal{C}\subset \mathrm{mod}A$ be a fixed $2$-cluster-tilting subcategory.
\begin{definition}Define a subcategory $\mathcal{X}\subseteq\mathcal{C}$ to be \emph{$2$-contravariantly finite in $\mathcal{C}$} if for any $M\in\mathcal{C}$ there exists a right $\mathcal{X}$-approximation $X_1\rightarrow M$ and object $X_2\in\mathcal{X}$ and a sequence 
$ X_2\rightarrow X_1\rightarrow M$ on which $\mathrm{Hom}(C,-)$ is exact for all $C\in \mathcal{C}$. 

Dually, define a subcategory $\mathcal{X}\subseteq\mathcal{C}$ to be \emph{$2$-covariantly finite in $\mathcal{C}$} if for any $M\in\mathcal{C}$ there exists a left $\mathcal{X}$-approximation $M\rightarrow X_1$ and object $X_2\in\mathcal{X}$ and a sequence 
$M \rightarrow X_1\rightarrow X_2$ on which $\mathrm{Hom}(-,C)$ is exact for all $C\in \mathcal{C}$.
A subcategory $\mathcal{X}\subseteq\mathcal{C}$ will be said to be \emph{$2$-functorially finite in $\mathcal{C}$} if it is both $2$-contravariantly finite in $\mathcal{C}$ and $2$-covariantly finite in $\mathcal{C}$. 
\end{definition}

\begin{lemma}
Let $\mathcal{X}\subseteq\mathcal{C}$ be functorially finite. Then $\mathcal{X}$ is 2-covariantly finite in $\mathcal{C}$ if and only if it is 2-contravariantly finite in $\mathcal{C}$.
\end{lemma}

\begin{proof}
It suffices to show any subcategory $\mathcal{X}$ 2-contravariantly finite in $\mathcal{C}$ is also $2$-covariantly finite in $\mathcal{C}$. So suppose $\mathcal{X}$ is 2-contravariantly finite in $\mathcal{C}$. Let $f:M\rightarrow X_0$ be a left $\mathcal{X}$ approximation and let $X_1^\prime \rightarrow X_0^\prime \rightarrow \mathrm{coker}f$ be a right $\mathcal{X}$-approximation. We have that $X_0\rightarrow \mathrm{coker}f$ factors through $X_1^\prime$, and since $X_0^\prime \not\cong\mathrm{coker}f$, this induces an additional non-zero morphism $M\rightarrow X_0^\prime$. The sequence $M\rightarrow X_0^\prime \rightarrow \mathrm{coker}f$ is now zero, and hence induces a morphism $g:M\rightarrow X_1^\prime$. Since $f:M\rightarrow X_0$ is a left $\mathcal{X}$-approximation, we have that $g$ factors through $X_0$. But this means $X_0\rightarrow \mathrm{coker}f$ factors through $X_1^\prime\rightarrow X_0^\prime$ and is hence zero, a contradiction. This is illustrated below:
$$\begin{tikzcd}
M \arrow{r}{f}\arrow[dotted]{d}{g}&X_0\arrow{r}\arrow[dotted]{d}\arrow[dotted]{dl}&\mathrm{coker}f\arrow[equals]{d}\\
X_1^\prime \arrow{r}&X_0^\prime\arrow{r}&\mathrm{coker}f
\end{tikzcd}$$
\end{proof}

Recall that for a finite-dimensional algebra $A$, then a \emph{torsion pair in $\mathrm{mod}A$} consists of two subcategories $(\mathcal{T},\mathcal{F})$ such that 
\begin{itemize}
\item $\mathrm{Hom}_A(T,F)=0$ for all $T\in\mathcal{T}$, $F\in\mathcal{F}$.
\item $\mathcal{T}=\{X\in\mathrm{mod}A|\mathrm{Hom}_A(X,F)=0\ \forall\ F\in\mathcal{F}\}$.
\item  $\mathcal{F}=\{Y\in\mathrm{mod}A|\mathrm{Hom}_A(T,Y)=0\ \forall\ T\in\mathcal{T}\}$.
\end{itemize}

We may now define one of our primary objects of study.
\begin{definition}
A \emph{2-functorially-finite torsion pair in $\mathcal{C}$} consists of two subcategories $(\mathcal{T},\mathcal{F})$ each $2$-functorially finite in $\mathcal{C}$ and such that 
\begin{itemize}
\item $\mathrm{Hom}_A(T,F)=0$ for all $T\in\mathcal{T}$, $F\in\mathcal{F}$.
\item $\mathcal{T}=\{X\in\mathcal{C}|\mathrm{Hom}_A(X,F)=0\ \forall\ F\in\mathcal{F}\}$.
\item  $\mathcal{F}=\{Y\in\mathcal{C}|\mathrm{Hom}_A(T,Y)=0\ \forall\ T\in\mathcal{T}\}$.
\end{itemize}
\end{definition}
For a $2$-functorially-finite torsion pair $(\mathcal{T},\mathcal{F})$ in $\mathcal{C}$, we say $\mathcal{T}$ is a \emph{torsion class} and $\mathcal{F}$ a \emph{torsion-free class}.

\begin{lemma}\label{torspair}
For any 2-functorially-finite torsion pair $(\mathcal{T},\mathcal{F})$ in $\mathcal{C}$ then $(\mathrm{Fac}\mathcal{T},\mathrm{Sub}\mathcal{F})$ is a torsion pair in $\mathrm{mod}A$.
\end{lemma}

\begin{proof}
Suppose $(\mathcal{T},\mathcal{F})$ is a 2-functorially-finite torsion pair in $\mathcal{C}$. Clearly any $X\in\mathrm{mod}A$ satisfies that $\mathrm{Hom}_A(\mathrm{Fac}\mathcal{T},X)=0$ if and only if $\mathrm{Hom}_A(\mathcal{T},X)=0$. Let $f:X\rightarrow C_X$ be a left $\mathcal{C}$-approximation of $X$, moreover $f$ must be injective. So $\mathrm{Hom}_A(\mathcal{T},X)=0$ if and only if $\mathrm{Hom}_A(\mathcal{T},C_X)=0$ and hence whenever $C_X\in\mathcal{F}$ and $X\in\mathrm{Sub}\mathcal{F}$. The result follows.
\end{proof}

\begin{proposition}\label{prop1}
Let $A$ be a finite-dimensional algebra and $\mathcal{C}\subseteq \mathrm{mod}A$ a 2-cluster-tilting subcategory. Let $\mathcal{T},\mathcal{F}\subseteq\mathcal{C}$ be 2-functorially-finite subcategories in $\mathcal{C}$. The following are equivalent.
 \begin{enumerate}[(i)]
\item  $(\mathcal{T},\mathcal{F})$ is a 2-functorially-finite torsion pair in $\mathcal{C}$.
\item For every $M\in\mathcal{C}$ there is an exact sequence $$T_M\rightarrow M\rightarrow F_M$$ such that $M\rightarrow F_M$ is a left $\mathcal{F}$-approximation and $T_M\rightarrow M$ is a right $\mathcal{T}$-approximation.
\end{enumerate}
\end{proposition}

\begin{proof}
(i)$\implies$ (ii): Given $M\in\mathcal{C}$, and a 2-functorially-finite torsion pair $(\mathcal{T},\mathcal{F})$ in $\mathcal{C}$, if follows from Lemma \ref{torspair} that $(\mathrm{Fac}\mathcal{T}, \mathrm{Sub}\mathcal{F})$ is a torsion pair in $\mathrm{mod}A$. A classical property of torsion pairs (see for example \cite[Proposition V1.1.5]{ass}) there is an exact sequence in $\mathrm{mod}A$: $0\rightarrow T\rightarrow M\rightarrow F \rightarrow 0$ for some $T\in\mathrm{Fac}\mathcal{T} $ and $F\in\mathrm{Sub}\mathcal{F}$. By taking appropriate approximations we obtain the result. 

(ii)$\implies$ (i) This is trivial.
\end{proof}

We may characterise torsion classes in $\mathcal{C}$.

\begin{proposition}\label{prop2}
For a 2-functorially-finite subcategory $\mathcal{T}\subseteq \mathcal{C}$, the following are equivalent:
 \begin{enumerate}[(i)]
\item There is an inclusion $(\mathrm{Fac}\mathcal{T}\cap\mathcal{C})\subseteq\mathcal{T}$ and for every $d$-exact sequence in $\mathcal{C}$ with $T_0,T_3\in\mathcal{T}$:$$0\rightarrow T_0\rightarrow X\rightarrow Y\rightarrow T_3\rightarrow 0$$ 
there exists a $d$-pushout diagram with exact rows:
$$\begin{tikzcd} 0 \arrow{r} & T_0^\prime  \arrow{r}\arrow{d} & T_1^\prime \arrow{r}\arrow{d} & T_2^\prime\arrow{r}\arrow{d} &T_3 \arrow[equals]{d}\arrow{r}&0\\
0 \arrow{r} & T_0\arrow{r} & X\arrow{r} & Y\arrow{r} &T_3\arrow{r}&0
\end{tikzcd}$$
such that $T_0^\prime,T_1^\prime,T_2^\prime\in\mathcal{T}$.
\item $\mathcal{T}$ is the torsion part of a 2-functorially-finite torsion pair in $\mathcal{C}$.
\end{enumerate}
\end{proposition}

\begin{proof}

(i)$\implies$ (ii)
For an arbitrary $M\in\mathcal{C}$, let $g:T_M\rightarrow M$ be a right $\mathcal{T}$-approximation. By Proposition \ref{prop1} it suffices to show that $\mathrm{Hom}_A(\mathcal{T},\mathrm{coker}g)=0$. 
So assume there exists a morphism $T_3\rightarrow \mathrm{coker}g$ for some $T_3 \in \mathcal{T}$. Then there exist $X,Y\in \mathcal{C}$ and a pullback diagram by Remark \ref{pullre}

$$\begin{tikzcd} 0 \arrow{r} & T_0 \arrow{r}\arrow[equal]{d} & X\arrow{r}\arrow{d} & Y\arrow{r}\arrow{d} &T_3 \arrow{d}\arrow{r}&0\\
0 \arrow{r} & T_0\arrow{r} & T_M\arrow{r} & M\arrow{r} &\mathrm{coker}g\arrow{r}&0
\end{tikzcd}$$
By assumption there exist some $T_0^\prime, T_1^\prime, T_2^\prime\in \mathcal{T}$ making the following diagram commute:

$$\begin{tikzcd} 0 \arrow{r} & T_0^\prime \arrow{r}\arrow{d} & T_1^\prime\arrow{r}\arrow{d} & T_2^\prime\arrow{r}\arrow{d} &T_3 \arrow[equal]{d}\arrow{r}&0\\
0 \arrow{r} & T_0 \arrow{r}\arrow[equal]{d} & X\arrow{r}\arrow{d} & Y\arrow{r}\arrow{d} &T_3 \arrow{d}\arrow{r}&0\\
0 \arrow{r} & T_0\arrow{r} & T_M\arrow{r} & M\arrow{r} &\mathrm{coker}g\arrow{r}&0
\end{tikzcd}$$
Since the composition $T_2^\prime\twoheadrightarrow T_3\rightarrow \mathrm{coker}g$ is non-zero there exists a non-zero morphism $T_2^\prime \rightarrow M$, which by assumption factors through $T_M$. But this implies the above non-zero morphism $T_2^\prime\rightarrow \mathrm{coker}g$ factors through $T_M\rightarrow M\rightarrow \mathrm{coker}g$, a contradiction.

(ii)$\implies$ (i)
Clearly $\mathcal{T}$ is closed under factor modules in $\mathcal{C}$. Now let 
$$(*):\ 0\rightarrow T_0\rightarrow X\rightarrow Y\rightarrow T_3\rightarrow 0$$  be $d$-exact with $T_0,T_3\in \mathcal{T}$. By Proposition \ref{prop1} there exist $T_X,T_Y\in\mathcal{T}$ and $F_X,F_Y \in\mathcal{F}$ and exact sequences
$$T_X\rightarrow X\rightarrow F_X,$$
$$T_Y\rightarrow Y\rightarrow F_Y.$$
Applying $\mathrm{Hom}_A(-,F)$ to $(*)$ implies an isomorphism $\mathrm{Hom}_A(Y,F)\cong \mathrm{Hom}_A(X,F)$ for any $F\in \mathrm{Sub}\mathcal{F}$;  this implies $\mathrm{Im}(X\rightarrow F_X)\cong\mathrm{Im}(Y\rightarrow F_Y)$. We are in the situation of Lemma \ref{topushout}, so there exist $T_X^\prime, T_Y^\prime, T^\prime \in\mathcal{T}$ such that there there is a pullback diagram

$$\begin{tikzcd}
&&0\arrow{d}&0\arrow{d}\\
&0\arrow{r}&T^\prime_X\arrow{r}\arrow{d}& T^\prime_Y \arrow{dr}\arrow{d}\\
0\arrow{r}&T_0 \arrow{r}\arrow[equal]{d}&T_X \arrow{d}\arrow{r}&T_Y\arrow{r}\arrow{d}&T^\prime\arrow{r}\arrow{d}&0\\
0\arrow{r} &T_0 \arrow{r}&X \arrow{r}&Y\arrow{r}&T_3\arrow{r}\arrow{d}&0\\
&&&&0
\end{tikzcd}$$ 
This induces the required $d$-pushout diagram

$$\begin{tikzcd}
0\arrow{r}& T_0\oplus T_X^\prime \arrow{r}\arrow{d}& T_X\oplus T_Y^\prime \arrow{r}\arrow{d}& T_Y\arrow{r}\arrow{d}& T_3\arrow{r}\arrow[equals]{d}& 0\\
0\arrow{r}& T_0 \arrow{r}&X \arrow{r}& Y\arrow{r}& T_3\arrow{r}& 0\\
\end{tikzcd}$$
since we know $0\rightarrow T_X^\prime \rightarrow T_X\oplus T_Y^\prime\rightarrow X\oplus T_Y\rightarrow Y\rightarrow 0$ is exact. 
\end{proof}

Recall an object $T\in\mathcal{C}$ is $\tau$-tilting \cite{adir} if 
\begin{itemize}
\item $\mathrm{Hom}_A(T,\tau T)=0$
\item $|T|=|A|$, where $|\cdot|$ denotes the number of indecomposable summands.
\end{itemize}
If $T$ is $\tau$-tilting as an $A/\langle e\rangle$-module for some idempotent $e$, then $T$ is \emph{support $\tau$-tilting}. This can be generalised as follows, using the generalised tilting theory of \cite{ha}, \cite{miya}: recall that an $A$-module $T$ is a \emph{$d$-tilting module} if:

\begin{enumerate}
\item $\mathrm{proj.dim}(T)\leq d$.
\item $\mathrm{Ext}^i_A(T,T)=0$ for all $0<i\leq d$.
\item there exists an exact sequence $$0\rightarrow A\rightarrow T_0\rightarrow T_1\rightarrow \cdots \rightarrow T_d\rightarrow 0$$ 
where $T_0,\ldots, T_d\in \mathrm{add}T$. 
\end{enumerate}
We may now define the latter of our primary objects of study.
\begin{definition}
An object $T\in\mathcal{C}$ is \emph{$\tau_2$-tilting} if 
\begin{itemize}
\item $\mathrm{Hom}_A(T,\tau_2 T)=0$
\item There exists an exact sequence
$$0\rightarrow A\rightarrow T_0\rightarrow T_1\rightarrow T_2\rightarrow 0$$ such that $T_0,T_1,T_2\in\mathrm{add}T$.
\end{itemize}
If $T$ is $\tau_2$-tilting as an $A/\langle e\rangle$-module for some idempotent $e$, then we say that $T$ is \emph{support $\tau_2$-tilting}.
\end{definition}
We now show that support $\tau_2$-tilting modules are indeed $2$-tilting. Recall that an $A$-module $T$ is \emph{faithful} if its right annihilator $\mathrm{ann}T$ vanishes.
\begin{lemma}[c.f. Lemma IV.2.7 of \cite{ass}; c.f Lemma VII.5.1 of \cite{ass}; c.f. Proposition 2.2(b)  of \cite{adir}]\label{adirlemma}
Let $A$ be a finite-dimensional algebra with 2-cluster-tilting subcategory $\mathcal{C}\subseteq \mathrm{mod}A$. Then
\begin{enumerate}[(i)]
\item For any $T\in\mathcal{C}$, $\mathrm{proj.dim}T\leq 2$ if and only if $\mathrm{Hom}_A(DA,\tau_2T)=0$.
\item Let $T\in\mathcal{C}$ be a faithful $A$-module. If $\mathrm{Hom}_A(T,\tau_2T)=0$, then $\mathrm{proj.dim}T\leq 2$.
\item Any $\tau_2$-tilting $A$-module $T$ is a tilting $(A/\mathrm{ann}T)$-module.
\end{enumerate}
\end{lemma}
\begin{proof}
$(i)$: Apply the left exact functor $\nu^{-1}=\mathrm{Hom}_A(DA,-)$ to the exact sequence
$$0\rightarrow \tau_2M\rightarrow \nu P_2\rightarrow \nu P_1\rightarrow \nu P_0\rightarrow \nu M\rightarrow 0$$ to obtain an exact sequence $$0\rightarrow \nu^{-1}\tau_2M\rightarrow \nu^{-1}\nu P_2\rightarrow \nu^{-1}\nu P_1\rightarrow \nu^{-1}\nu P_0\rightarrow \nu^{-1}\nu M\rightarrow 0$$
It follows that $\mathrm{Hom}_A(DA,\tau_2 M)=\nu^{-1}\tau_2M$ vanishes if and only if $\mathrm{proj.dim}M\leq 2$.

$(ii)$: It is known (see \cite[V1.2.2]{ass}) that an $A$-module $T$ is faithful if and only if $DA$ is generated by $T$. So let $T^i\twoheadrightarrow DA$ be a surjection. Applying the functor $\mathrm{Hom}_A(-,\tau_2T)$ results in a monomorphism $\mathrm{Hom}_A(DA,\tau_2 T)\hookrightarrow \mathrm{Hom}_A(T,\tau_2T)$. The result now follows from part (i) above.

$(iii)$   Note that for any idempotent ideal $\langle e\rangle$ of $A$, and any $M,N\in\mathrm{mod}A$ there is a natural inclusion $\mathrm{Ext}^2_{A/\langle e \rangle}(M,N)\hookrightarrow \mathrm{Ext}^2_{A}(M,N)$. In this case let $\langle e\rangle:=\mathrm{ann}T$. Since  $\mathrm{Hom}_A(T,\tau_2 T)=0$, Proposition \ref{aslemma} implies $\mathrm{Ext}^2_A(T,\mathrm{Fac}T)=0$, and it follows that $\mathrm{Ext}^2_{A/\langle e\rangle}(T,\mathrm{Fac}T)=0$. By Proposition \ref{aslemma} once more, we have $\mathrm{Hom}_{A/\langle e\rangle}(T,\tau_2T)=0$. Since $T$ is faithful as an $(A/\mathrm{ann}T)$-module, it follows from part $(ii)$ that $\mathrm{proj.dim}T\leq d$ and that $T$ is $2$-tilting as an $(A/\mathrm{ann}T)$-module.
\end{proof}
\begin{lemma}[Happel \cite{ha}]\label{happel}
Let $A$ be a finite-dimensional algebra and $T$ be a $d$-tilting $A$-module. Assume that $M\in \mathrm{mod}(A)$ satisfies $\mathrm{Ext}_A^i(T,M)=0$ for all $i>0$. Then there exists an exact sequence $$0\rightarrow T_m\rightarrow \cdots \rightarrow T_1\rightarrow T_0\rightarrow M\rightarrow 0$$ such that $T_j\in\mathrm{add}(T)$ for all $0\leq j\leq m$ and $m\leq \mathrm{gl.dim}A$.
\end{lemma}

We may now prove our main result.
\begin{customthm}{\ref{elso}}
An object $T\in\mathcal{C}$ is a support $\tau_2$-tilting module if and only if $\mathrm{Fac}\mathcal{T} \cap\mathcal{C}$ is the torsion part of a 2-functorially-finite torsion pair in $\mathcal{C}$.
\end{customthm}

\begin{proof}
Let $\mathcal{T}$ be a torsion class in $\mathcal{C}$, then by 2-functorial finiteness and closure under factor modules, there exists an exact sequence $A\rightarrow T_0\rightarrow T_1\rightarrow T_2\rightarrow 0$ such that $f:A\rightarrow T_0$ is a left $\mathcal{T}$-approximation and $T_1, T_2\in\mathcal{T}$. Clearly $\mathrm{Im}f\cong A/\mathrm{Ann}\mathcal{T}$.

 Now we need to prove that $T_i$ are $\mathrm{Ext}$-projective in $\mathcal{T}$ for all $0\leq i\leq 2$, i.e. that $\mathrm{Ext}^2_A(T_i,T)=0$ for all $T\in\mathcal{T}$.
First suppose $\mathrm{Ext}_A^2(T_0,T)\ne 0$. Then by Proposition \ref{prop2} there exists some $T_0^\prime\in\mathcal{T}$ and a non-split surjection $T_0^\prime\twoheadrightarrow T_0$. But then $f$ must factor through $T_0^\prime$, a contradiction. 

We next claim that there is a surjection $\mathrm{Ext}^2_A(T_2,T)\twoheadrightarrow \mathrm{Ext}_A^2(T_1,T)$ for any $T\in\mathcal{T}$. Let $X:=\mathrm{Im}(T_0\rightarrow T_1)$. This induces exact sequences
$$ \mathrm{Ext}^2_A(T_2,T)\rightarrow \mathrm{Ext}^2_A(T_1,T)\rightarrow \mathrm{Ext}^2_A(X,T)$$ 
$$\mathrm{Ext}^1_A(A/\mathrm{ann}\mathcal{T},T)\rightarrow \mathrm{Ext}^2_A(X,T)\rightarrow \mathrm{Ext}^2_A(T_0,T)=0.$$
Suppose $\mathrm{Ext}^1_A(A/\mathrm{ann}\mathcal{T},T)\ne 0$ and there exists a short exact sequence
$$0\rightarrow T\rightarrow E\rightarrow A/\mathrm{ann}\mathcal{T}\rightarrow 0.$$
But then $E$ must be an $A/\mathrm{ann}\mathcal{T}$-module, since both $T$ and $A/\mathrm{ann}\mathcal{T}$ are, meaning that this sequence splits, a contradiction. So $\mathrm{Ext}^1_A(A/\mathrm{ann}\mathcal{T},T)= 0 = \mathrm{Ext}^2_A(X,T)$, and there is a surjection $\mathrm{Ext}^2_A(T_2,T)\twoheadrightarrow \mathrm{Ext}_A^2(T_1,T)$ as claimed.

Now suppose $\mathrm{Ext}_A^2(T_2,T^\prime)\ne 0$ for some $T^\prime\in\mathcal{T}$. By Proposition \ref{prop2} there exists $T_0^\prime,T_1^\prime,T_2^\prime \in\mathcal{T}$ and a $d$-exact sequence $0\rightarrow T_0^\prime \rightarrow T_1^\prime \rightarrow T_2^\prime \rightarrow T_2\rightarrow 0$. 
By Lemma \ref{pushout} there exist $P,Q\in\mathcal{C}$ and a commutative diagram with exact rows and columns:
$$\begin{tikzcd}
&&&0\arrow{d}&0\arrow{d}\\
&&&A/\mathrm{ann}\mathcal{T}\arrow[equal]{r}\arrow{d}&A/\mathrm{ann}\mathcal{T}\arrow{d}\\
&& Q\arrow{r}\arrow{d}& Q\oplus T_0 \arrow{r}\arrow{d}& T_0\arrow{d}\\
0\arrow{r}&T_0^\prime \arrow{r}\arrow[equal]{d}&Q\oplus T_1^\prime \arrow{d}\arrow{r}&P\arrow{r}\arrow{d}&T_1\arrow{r}\arrow{d}&0\\
0\arrow{r} &T_0^\prime  \arrow{r}&T_1^\prime \arrow{r}&T_2^\prime \arrow{r}\arrow{d}&T_2\arrow{r}\arrow{d}&0\\
&&&0&0
\end{tikzcd}$$ 
Moreover $P,Q\in\mathcal{T}$ by assumption, since they have no common non-zero summands and we may replace $T_0^\prime$ if necessary. This implies a morphism $A\rightarrow Q$, which by assumption factors through $T_0$. But this implies the morphism $T_0\rightarrow T_1$ also factors through $Q$, and is hence the zero composition $Q\rightarrow P\rightarrow T_1$, a contradiction.


On the other hand, let $T$ be a support $\tau_2$-tilting module. We show that $\mathrm{Fac}T\cap\mathcal{C}$ satisfies the conditions of Proposition \ref{prop2}. Closure under factor modules is trivial, and 2-functorial finiteness in $\mathcal{C}$ follows from Lemma \ref{happel}. So let 
$$0\rightarrow T_0 \rightarrow X\rightarrow Y\rightarrow T_3\rightarrow 0$$ be a $d$-exact sequence in $\mathcal{C}$ with $T_0,T_3\in\mathrm{Fac}T$. By Lemma \ref{aslemma}, we have that $\mathrm{Hom}_A(T,\tau_2T)=0\iff \mathrm{Ext}^2_A(T,\mathrm{Fac}T)=0$, so $T_3$ is not in $\mathrm{add}T$. Lemma \ref{happel} implies there exists an exact sequence in $\mathrm{Fac}T$:
$0\rightarrow T_0^\prime \rightarrow T_1^\prime \rightarrow T_2^\prime \rightarrow T_3\rightarrow 0.$ By Lemma \ref{pushout} there exist  $P,Q\in\mathcal{C}$ to form a commutative diagram with exact rows and columns:

$$\begin{tikzcd}
&&&0\arrow{d}&0\arrow{d}\\
&&&T_0^\prime \arrow[equal]{r}\arrow{d}& T_0^\prime\arrow{d}\\
&& Q\arrow{r}\arrow{d}& Q\oplus T_1^\prime \arrow{r}\arrow{d}& T_1^\prime\arrow{d}\\
0\arrow{r}&T_0 \arrow{r}\arrow[equal]{d}&Q\oplus X \arrow{d}\arrow{r}&P\arrow{r}\arrow{d}&T_2^\prime\arrow{r}\arrow{d}&0\\
0\arrow{r} &T_0 \arrow{r}&X \arrow{r}&Y\arrow{r}\arrow{d}&T_3\arrow{r}\arrow{d}&0\\
&&&0&0
\end{tikzcd}$$ 
 Applying 
Lemma \ref{aslemma} once more, we find $\mathrm{Ext}^2_A(T_2^\prime,T_0)=0$. This implies the diagram splits, $Q\cong T_0$, $P\cong T_2^\prime\oplus X$ and that we are in the situation of Lemma \ref{topushout}, with a commutative diagram with exact rows and columns:
$$\begin{tikzcd}
&&0\arrow{d}&0\arrow{d}\\
&0\arrow{r}&T_0^\prime \arrow{r}\arrow{d}& T_1^\prime \arrow{dr}\arrow{d}\\
0\arrow{r}&T_0 \arrow{r}\arrow[equal]{d}&T_0\arrow{d}\arrow{r}&T_2^\prime\arrow{r}\arrow{d}&T_2^\prime \arrow{r}\arrow{d}&0\\
0\arrow{r} &T_0 \arrow{r}&X \arrow{r}&Y\arrow{r}&T_3\arrow{r}\arrow{d}&0\\
&&&&0
\end{tikzcd}$$ 

Hence we find the required $d$-pushout diagram
$$\begin{tikzcd}
0\arrow{r}& T_0\oplus T_0^\prime \arrow{r}\arrow{d}& T_0\oplus T_1^\prime \arrow{r}\arrow{d}& T_2^\prime\arrow{r}\arrow{d}& T_3\arrow{r}\arrow[equals]{d}& 0\\
0\arrow{r}& T_0 \arrow{r}&X \arrow{r}& Y\arrow{r}& T_3\arrow{r}& 0\\
\end{tikzcd}$$
since we know $0\rightarrow T_0^\prime \rightarrow T_0\oplus T_1^\prime\rightarrow X\oplus T_2^\prime\rightarrow Y\rightarrow 0$ is exact. 
\end{proof}



\bibliographystyle{amsplain}
\bibliography{citations}

\end{document}